\documentclass[10pt]{amsart}
\usepackage{amsmath}
\usepackage{epsfig}
\usepackage{graphicx}
\usepackage{eufrak}
\usepackage{dsfont}
\usepackage[english]{babel}
\usepackage{color}

\usepackage{palatino, mathpazo}
\usepackage{amscd}      
\usepackage{amssymb}
\usepackage{dsfont}
\usepackage{xypic}      
\LaTeXdiagrams          
\usepackage[all,v2]{xy}
\xyoption{2cell} \UseAllTwocells \xyoption{frame} \CompileMatrices
\usepackage{latexsym}
\usepackage{epsfig}
\usepackage{enumerate}

\usepackage{latexsym}
\usepackage{epsfig}
\usepackage{amsfonts}
\usepackage{enumerate}
\usepackage{mathrsfs}

\newbox\mybox
\def\overtag#1#2#3{\setbox\mybox\hbox{$#1$}\hbox to
  0pt{\vbox to 0pt{\vglue-#3\vglue-\ht\mybox\hbox to \wd\mybox
      {\hss$\ss#2$\hss}\vss}\hss}\box\mybox}
\def\undertag#1#2#3{\setbox\mybox\hbox{$#1$}\hbox to 0pt{\vbox to
    0pt{\vglue#3\vglue\ht\mybox\hbox to \wd\mybox
      {\hss$\ss#2$\hss}\vss}\hss}\box\mybox}
\def\lefttag#1#2#3{\hbox to 0pt{\vbox to 0pt{\vss\hbox to
      0pt{\hss$\ss#2$\hskip#3}\vss}}#1}
\def\righttag#1#2#3{\hbox to 0pt{\vbox to 0pt{\vss\hbox to
      0pt{\hskip#3$\ss#2$\hss}\vss}}#1}
\let\ss\scriptstyle

\def\Dot{\lower.2pc\hbox to 2pt{\hss$\bullet$\hss}}
\def\Circ{\lower.2pc\hbox to 2pt{\hss$\circ$\hss}}
\def\Vdots{\raise5pt\hbox{$\vdots$}}
\newcommand\lineto{\ar@{-}}
\newcommand\dashto{\ar@{--}}
\newcommand\dotto{\ar@{.}}

\allowdisplaybreaks[3]

\newtheorem{prop}{Proposition}[section]

\newtheorem{cor}[prop]{Corollary}
\newtheorem{thm}[prop]{Theorem}

\newtheorem{defn}[prop]{Definition}

            {\nolinebreak $\Box$ \end{trivlist}}

\newcommand{\noprint}[1]{}

\renewcommand{\tilde}{\widetilde}

\newcommand{\sm}{\mbox{\tiny sm}}

\newcommand{\Hom}{\mbox{Hom}}

\newcommand{\XX}{{\mathfrak X}}

\newcommand{\YY}{{\mathfrak Y}}

\newcommand{\MM}{{\mathfrak M}}

\newcommand{\zz}{{\mathbb Z}}

\newcommand{\qq}{{\mathbb Q}}
\newcommand{\pp}{{\mathbb P}}
\newcommand{\cc}{{\mathbb C}}
\newcommand{\rr}{{\mathbb R}}

\newcommand{\Gm}{{{\mathbb G}_{\mbox{\tiny\rm m}}}}

\newcommand{\sT}{{\mathcal T}}

\newcommand{\sO}{{\mathcal O}}

\newcommand{\sX}{{\mathcal X}}

\DeclareMathOperator{\id}{id}

\DeclareMathOperator{\Aut}{Aut}

\DeclareMathOperator{\lci}{lci}

\DeclareMathOperator{\rad}{rad}
\DeclareMathOperator{\perm}{perm}

\newcommand{\rk}{\mathop{\rm rk}}

\newcommand{\tr}{\mathop{\rm tr}\nolimits}

\renewcommand{\Im}{\mathop{\rm Im}}

\newcommand{\sign}{\mathop{\rm sign}}

\newcommand{\cok}{\mathop{\rm cok}}

\newcommand{\spec}{\mathop{\rm Spec}\nolimits}

\newcommand{\tor}{\mathop{\rm tor}\nolimits}

\numberwithin{equation}{subsection}

\newcommand {\mat}      [1] {\left(\begin{array}{#1}}
\newcommand {\rix}          {\end{array}\right)}

\title[Smoothing of surface singularities]{Smoothing of  surface singularities via equivariant smoothing of lci covers}
\author{Yunfeng Jiang}
\address{Department of Mathematics\\ University of Kansas\\ 405 Snow Hall 1460 Jayhawk Blvd\\Lawrence KS 66045 USA} 
\email{y.jiang@ku.edu}

\begin{document}
\sloppy \maketitle
\begin{abstract}
We  provide some  results  of the smoothing of surface singularities by Looijenga-Wahl and study smoothing of isolated surface singularities induced by equivariant smoothing of locally complete intersection ($\lci$) singularities.  We classify the situation where the smoothing of  a simple elliptic singularity, a cusp singularity  or its cyclic quotient is induced by the equivariant smoothing of the $\lci$ covers. 
\end{abstract}

\maketitle


\section{Introduction}

\subsection{Results of Looijenga-Wahl}
Let $(X,0)$ be a germ of an  isolated  surface singularity.  A  smoothing of $(X,0)$ is a flat analytic morphism $f: (\XX,0)\to (\Delta,0)$, where $(\Delta, 0)$ is an open disc such that the special fiber 
$(f^{-1}(0),0)\cong (X,0)$ and the generic fiber $f^{-1}(t)$ is smooth for $t\neq 0$. The study of smoothing of surface singularities has a long history, see \cite{Pinkham}, \cite{Wahl3}, \cite{Wahl}, \cite{Looijenga-Wahl}, \cite{Kollar-Shepherd-Barron}.  We can not list all the references for the surface smoothing and we suggest the readers for the references in the above papers. The smoothing of isolated surface singularities has interesting discoveries on the topology and geometry of the singularity.  Some special type of isolated surface singularities (semi-log-canonical singularities) and their smoothing have important applications in the KSBA compactification of moduli space of general type surfaces \cite{Kollar-Shepherd-Barron}. 

We briefly recall the topological invariants for a smoothing  $f: (\XX,0)\to (\Delta,0)$.  We assume that $\XX$ is a Stein space with a partial boundary such that 
$f|_{\XX\setminus \{0\}}: \XX\setminus \{0\}\to \Delta\setminus \{0\}$ is the Milnor fibration. 
The generic fiber $M:=f^{-1}(t)$ is a real 4-manifold with boundary which is called the Milnor fibre of the smoothing.  The space $M$ has the homotopy type of a CW complex of (real) dimension $\le 2$, and 
its boundary $\partial M$ is diffeomorphic to the link $L_X$ of the surface singularity $(X,0)$.  
The link $L_X$ is a smooth 3-dimensional oriented manifold.  Also if $\pi: \widetilde{X}\to X$ is the minimal resolution of $X$, then we can regard the link $L_X$ as the boundary of both 
$M$ and $\widetilde{X}$.

The homology groups $H_i(M), H_i(L_X), H_i(M, L_X)$ are the important topological invariants for the smoothing $f$ which fit into the following exact sequences:
\begin{equation}\label{eqn_eact_sequence}
\cdots\rightarrow H_2(M)\rightarrow H_2(M, L_X)\stackrel{\partial}{\rightarrow} H_1(L_X)\rightarrow H_1(M)\rightarrow H_1(M, L_X)\rightarrow\cdots
\end{equation}

The second homology group $H_2(M)$ is equipped with the intersection product which is symmetric. 
The group $H_2(\widetilde{X})$ is free and has a natural set  of fundamental classes of the irreducible exceptional curves as generators, and the intersection pairing is negative-definite. 
For  $H_2(M)$,  we have the Sylvester invariants $(\mu_0,\mu_+, \mu_-)$ of the real corresponding pairing in \cite{Steenbrink}, where 
$\mu_0+\mu_{+}+\mu_{-}=\rk H_2(M)$.
Looijenga-Wahl \cite{Looijenga-Wahl} constructed natural and compatible {\em quadratic functions} on $H_2(M)$, $H_2(\widetilde{X})$, and on the torsion subgroup $H_2(L_X)_{\tor}$ of $H_2(L_X)$.
The notations are denoted by $Q_M, Q_{\widetilde{X}}$ and $q$ respectively. 
A detailed review is in \S \ref{sec_quadratic_function}.

We have the following result in \cite{Looijenga-Wahl}:

\begin{thm}\label{thm_Looijenga-Wahl_intro}(\cite[Theorem 4.5]{Looijenga-Wahl})
Suppose that $M$ is the Milnor fiber of a smoothing of an isolated surface singularity $(X,0)$ with link $L_X$. Then the quadratic function $Q_M$ on $H_2(M)$ makes it an ordinary quadratic lattice 
with the associated non-degenerate lattice $(\overline{H}_2(M), Q_M)$.  The corresponding Discriminant Quadratic Form (DQF) is canonically isomorphic to 
$(I^{\perp}/I, q_I)$ where $I\subset H_1(L_X)_{\tor}$ is the $q$-isotropic subgroup $I:=\Im(\partial_{\tor}: H_2(M, L_X)_{\tor}\to H_1(L_X)_{\tor})$.  

Moreover, there exists an exact sequence 
$$0\to H_1(L_X)_{\tor}/I^{\perp}\rightarrow H_1(M)\rightarrow P\to 0$$
where $P$ is a quotient of $H_1(L_X)/H_1(L_X)_{\tor}\cong H_1(\widetilde{X})$. 
\end{thm}

\subsection{Smoothing component}\label{subsec_smoothing_component_intro}

Recall that a smoothing component of $(X,0)$ is by definition an irreducible component of the base space $\MM$ of a versal deformation of $(X,0)$ over which the generic fiber is smooth. 
We let 
$f: \XX\to \MM$
to represent the versal deformation space of $(X,0)$ and 
$$f: \XX\to \MM^{\sm}$$
the base space of all the smoothing components.  Let $D\subset \MM^{\sm}$ be the discriminant locus such that there is a fibration
$$f^{-1}(\MM^{\sm}\setminus D)\to \MM^{\sm}\setminus D.$$
Each connected component of $\MM^{\sm}\setminus D$ is of the form $\MM^{1}\setminus (\MM^{1}\cap D)$ for 
$\MM^{1}$ a smoothing component.   For any $s\in \MM^{1}$, we have 
$$H_2(\XX_s, \partial \XX_s)\stackrel{\partial_s}{\rightarrow} H_1(\partial \XX_s)\cong H_1(\partial X)\cong H_1(L_X)$$
and the image of its torsion part $H_2(\XX_s, \partial \XX_s)_{\tor}$ only depends on the connected component of $s\in \MM^{\sm}\setminus D$. Thus,
each smoothing component 
$S\subset \MM^{\sm}$ of $(X,0)$ determines a subgroup
$J(S)\subset H_1(L_X)$, $J(S)\cap H_1(L_X)_{\tor}=I(S)^{\perp}$, where 
$I(S)$ is $q$-isotropic for $q: H_1(L_X)_{\tor}\to \qq/\zz$ the quadratic function. 

Let us fix the isolated surface singularity $(X, 0)$.  Then we have invariants $H_1(L_X)$ and the quadratic function $q: H_1(L_X)_{\tor}\to \qq/\zz$. 
For any integer $\mu_+>0$, let 
$$\mathscr{S}(\mu_+, H_1(L_X), q):=\{(V, Q, I, i, s)\}/\sim$$
where the $5$-tuples $(V, Q, I, i, s)$ are the following:
\begin{enumerate}
\item $(V,Q)$ is an ordinary lattice whose maximal positive definite subspaces have dimension $\mu_+$;
\item  $s$ is a sign structure on $(V,Q)$;
\item  $I$ is an isotropic  subspace of $H_1(L_X)_{\tor}$;
\item 
$i: V^*/B(V)\to H_1(L_X)/I$ is an injective homomorphism with finite cokernel which induces an isomorphism $\overline{V}^{\#}/\overline{V}\to I^{\perp}/I$ of quadratic groups, where 
$B: V\to V^*$ denotes the adjoint of the bilinear part of $Q$.
\end{enumerate}
The equivalence relation is defined as follows.  Two  $5$-tuples $(V_1, Q_1, I_1, i_1, s_1)$ and $(V_2, Q_2, I_2, i_2, s_2)$ are equivalent if $I_1=I_2$ and there exists an isomorphism
$\Phi: (V_1, Q_1, I_1, i_1, s_1)\to (V_2, Q_2, I_2, i_2, s_2)$ of quadratic lattices with sign structure such that $i_2=\phi\circ i_1$ where $\phi$ is induced by $\Phi$. 
If $\mu_+=0$, we can ignore the sign structure and denote by $\mathscr{S}(0, H_1(L_X), q)$ the set of $4$-tuples $(V, Q, I, i)$ modulo equivalence relation. 
We denote by  $|\MM^{\sm}|$ the set of smoothing components of the normal surface singularity $(X,0)$.
Then \cite{Looijenga-Wahl} showed that there is a natural map
\begin{equation}\label{eqn_map_smoothing_component}
\Pi: |\MM^{\sm}|\longrightarrow  \mathscr{S}(\mu_+, H_1(L_X), q)
\end{equation}
which sends a smoothing component $\MM^{1}\subset \MM^{\sm}$ to the data $(H_2(M), Q_M, I, i, s)$ constructed from $f: (\XX,0)\to (\Delta,0)$.

\subsection{Equivariant smoothing}\label{subsec_equivariant_smoothing_intro}

In this paper we are interested in the smoothing of $(X,0)$ which are induced by  equivariant  smoothing of locally complete intersection ($\lci$) surface singularities. 
Let $f: (\XX,0)\to (\Delta,0)$ be a smoothing of the surface singularity $(X,0)$.   From \cite[Proposition 5.2]{Looijenga-Wahl}, if there is a finite unramified covering 
$$\pi: \widetilde{M}\to M$$
for the Milnor fiber $M$, then it can be extended to a finite morphism 
$\pi: \YY\to \XX$ unramified over $\XX\setminus \{0\}$ and $\YY$ is a normal Stein space with a partial boundary.  Also 
$\pi^{-1}(0)=\{0\}$ and $(Y,0):=(\pi^{-1}(X),0)$ is a surface singularity. The composition map 
$f\circ \pi: (\YY,0)\to (\Delta,0)$ is a smoothing of $(Y,0)$ with $\widetilde{M}$ its corresponding Milnor fiber. 
Let $G$ be the  finite transformation group of the unramified  covering $\widetilde{M}\to M$.  Then $G$ acts on $\YY$ only fixing $\{0\}$, and thus 
$\widetilde{f}=f\circ\pi: (\YY,0)\to (\Delta,0)$ is a $G$-equivariant smoothing of $(Y,0)$ with quotient the smoothing $f: (\XX,0)\to (\Delta,0)$. 

We are interested in the case that $(Y,0)$ is an $\lci$ surface singularity, so that $(\YY,0)$ is also an $\lci$ threefold singularity.   
From \cite{Le}, \cite{Hamm},  its link $L_{\YY}$ is simply connected, which implies that 
$\pi_1(\widetilde{M})=0$. This implies that $\pi_1(M)=G$. 
 In this case there is no 
nontrivial $I\subset  H_1(L_Y)_{\tor}$ in the above construction.  From the construction of the unramified cover 
$$\pi: \widetilde{M}\to M$$
which induces an unramified $G$-cover $\pi: L_Y\to L_X$ on the boundaries.  There are morphisms 
$$\pi^{!}: H_i(M)\to H_i(\widetilde{M})$$
which sends a cycle in $M$ to its preimage in $\widetilde{M}$ (which is  also a cycle) satisfying 
\begin{equation}\label{eqn_conditions}
\pi_*\pi^{!}(x)=d x; \quad  \text{and when~} i=2, \pi^{!}(x)\cdot \pi^{!}(y)=d (x\cdot y)
\end{equation} 
Here $d=|G|$ is the order of $G$.  Similar formula holds for  $\pi: L_Y\to L_X$.
Also in this case we have 
$$DQF(H_2(\widetilde{M}), \widetilde{Q})\cong (H_1(L_Y)_{\tor}, \tilde{q}).$$
We let $\mathscr{S}(\widetilde{\mu}_{+},H_1(L_Y), \tilde{q})$ be the set of all such $4$-tuples $(H_2(\widetilde{M}), \widetilde{Q}, \tilde{i}, \tilde{s})$. 

For any integer $\mu_+>0$, to shorten notations let $A_Y:=H_1(L_Y)$, and  $A_X:=H_1(L_X)$, we set 
$$\mathscr{S}\left((\widetilde{\mu}_{+},A_Y, \tilde{q})\rightarrow(\mu_+, A_X, q)\right):=\{(\widetilde{V}, \widetilde{Q}, \tilde{i}, \tilde{s})\to (V, Q, I, i, s)\}/\sim$$
where 
\begin{enumerate}
\item the $4$-tuples $(\widetilde{V}, \widetilde{Q}, \tilde{i}, \tilde{s})$ is same as the $5$-tuples in \S \ref{subsec_smoothing_component_intro} with only $I=0$;
\item the arrow $(\widetilde{V}, \widetilde{Q}, \tilde{i}, \tilde{s})\to (V, Q, I, i, s)$ means that there is maps 
$\pi^{!}: V\to H_i(\widetilde{V})$, and $\pi_{*}: \widetilde{V}\to V$ satisfying $\pi_*\pi^{!}(x)=d x$,  and  $\pi^{!}(x)\cdot \pi^{!}(y)=d (x\cdot y)$. 
Similar formula holds for $A_Y$ and $A_X$. 
\item all the maps above are compatible with the quadratic forms and the maps $i, \tilde{i}$ and the signs $s, \tilde{s}$.
\end{enumerate}

We denote by  $|\MM_{\lci}^{\sm}|$ the set of smoothing components of the normal surface singularity $(X,0)$ that can be obtained as the equivariant smoothing of 
$\lci$ covers.
We have the first main result:

\begin{thm}\label{thm_smoothing_component_lci_intro}
There is a natural map 
$$\Pi^{\lci}: |\MM_{\lci}^{\sm}|\longrightarrow  \mathscr{S}\left((\widetilde{\mu}_{+},A_Y, \tilde{q})\rightarrow(\mu_+, A_X, q)\right)$$
which sends an equivariant smoothing component $\MM_{\lci}^{1}\subset \MM_{\lci}^{\sm}$ to the data $(H_2(\widetilde{M}), \widetilde{Q}, \tilde{i}, \tilde{s})\to (H_2(M), Q_M, I, i, s)$ constructed above.
\end{thm}

We call any smoothing component in  $|\MM_{\lci}^{\sm}|$ the $\lci$ smoothing component.
If $(\Delta,0)\subset \MM_{\lci}^{\sm}$, we call the corresponding  smoothing $\widetilde{f}: (\YY,0)\to (\Delta,0)$ the $\lci$ smoothing lifting of the smoothing $f: \XX\to \Delta$ for $(X,0)$.

\subsection{Applications-simple elliptic singularities}\label{subsec_slc_intro}

The main result Theorem \ref{thm_smoothing_component_lci_intro} has applications to special type of surface singularities. 

In the KSBA compactification of moduli space of general type surfaces in \cite{Kollar-Shepherd-Barron},  slc surfaces are added to the boundary of the moduli space.  The slc surface singularities were classified in 
\cite[Theorem 4.24]{Kollar-Shepherd-Barron}, \cite[Theorem 9.13]{Kawamata}. 
Let $X$ be a slc surface.   Except $\lci$-singularities, quotient surface singularities,  the important isolated normal  slc surface singularities are log canonical surface singularities. 
According to the classification,  there are two classes of such isolated singularities depending on the local index. 
\begin{enumerate}
\item  simple elliptic singularities and cusps, these are Gorenstein singularities which have local index one;
\item  the $\zz_2, \zz_3, \zz_4, \zz_6$-quotients of  simple elliptic singularities  and the $\zz_2$-quotient of cusp singularities. These are $\qq$-Gorenstein singularities which have local index $r$ for $r=2,3,4,6$. 
\end{enumerate}

Recall that a simple elliptic singularity $(X,0)$ is an isolated singularity such that the exceptional curve of its minimal resolution $(\widetilde{X}, E)$ is a smooth elliptic curve.   The local embedded dimension of the singularity is given by 
$\max(3, -E\cdot E)$.   It is known from \cite{Laufer}, that the simple elliptic singularity $(X,0)$ is an $\lci$ singularity if the negative self-intersection $-E\cdot E\le 4$.  If $-E\cdot E\ge 5$, then $(X,0)$ is never $\lci$. 
Let $d:=-E\cdot E$ be the degree of $(X,0)$.  From \cite{Pinkham}, a simple elliptic singularity $(X,0)$ admits a smoothing if and only if $1\leq d\leq 9$. 

We have the following result:
\begin{thm}\label{thm_elliptic_singularity_intro}
Let $(X,0)$ be a simple elliptic surface singularity,  and $(\widetilde{X}, E)$ its minimal resolution. Then $(X,0)$ admits an $\lci$ smoothing lifting only when 
$1\le d\le 9$ and $d\neq 5, 6, 7$.
\end{thm}

Let $S$ be a smooth del Pezzo surface  of degree $d$ with $1\le d\leq 9$.  Then the affine cone 
$$C(S, -K_S)=\spec \oplus_{i=0}^{\infty}H^0(S, i(-K_S))$$
has the origin as  a threefold isolated singularity.  It is a smoothing of a simple elliptic singularity $(X,0)$ of degree $d$.  Therefore, there is a map 
$f: (C(S, -K_S),0)\to (\Delta,0)$ such that $(f^{-1}(0),0)\cong (X,0)$ is a cone over an elliptic curve.  Theorem \ref{thm_elliptic_singularity_intro} implies that 
only the cones $C(S, -K_S)$ of degree $d=8$ (resp. $d=9$) admits a nontrivial $\lci$ $\zz_2$-covering (resp. $\zz_3$-covering) $\widetilde{C(S, -K_S)}\to C(S, -K_S)$ 
such that $\widetilde{C(S, -K_S)}$ is a cone over a degree $d=4$ (resp. $d=3$) del Pezzo surface.

\subsection{Cusp singularities}

A cusp singularity $(X,0)$ is an isolated normal surface singularity such that the exceptional curve of its minimal resolution $(\widetilde{X}, E)$ is a circle of rational curves or a rational nodal curve.  Let 
$$E=E_1+\cdots+E_r\in |-K_{\widetilde{X}}|.$$
The analytic germ $(X,0)$ of the cusp singularity is uniquely determined by the self-intersections $E_i^2$ of $E$. 
When $n\ge 3$, we assume that $E_i\cdot E_{i\pm 1}=1$.  If $r=2$, $E$ is the union of two smooth rational curves that meet transversely at two distinct points. 
Since $E$ is contractible, Artin's criterion for contractibility implies that the intersection matrix $[E_i\cdot E_j]$ is negative-definite. 
We let $\mathbf{d}=[-d_1, \cdots, -d_r]$ be the self-intersection sequence. 

The cusp singularities exist as pairs which are the only two singular points of an Inoue surface, see \cite{Inoue}.  Let $(X,0)$ be a cusp singularity with resolution cycle $E$, then there exists a dual cusp 
$(X^\prime,0)$ with resolution cycle $E^\prime$.   Let $\mathbf{d}^\prime=[-d^\prime_1, \cdots, -d^\prime_s]$ be the self-intersection sequence of $(X^\prime, 0)$.  Then 
their relation can be found in \cite{Pinkham}, \cite{Jiang_cusp}.   For the smoothing of cusp singularity $(X,0)$,   Looijenga \cite{Looijenga2} conjectured that a cusp singularity 
$(X,0)$ is smoothable if and only if for its dual cusp $(X^\prime, 0)$, the resolution cycle $E^\prime$ is an anti-canonical divisor of a smooth rational surface $Z$.   We call $(Z,E^\prime)$ a Looijenga pair. 
Looijenga \cite{Looijenga2} proved the necessary condition.  The sufficient condition was proved in \cite{GHK15} using log Gromov-Witten theory.   Engel \cite{Engel} gave a birational geometry proof using type III degeneration of 
Looijenga pairs.

From \cite[Proposition 4.1]{NW}, a cusp singularity $(X,0)$ admits a  finite $\lci$ cover $\pi: (Y,0)\to (X,0)$ where $(Y,0)$ is an $\lci$ cusp. 
Let $G$ be the transformation group for $\pi: (Y,0)\to (X,0)$. 
Using the  theory of Looijenga-Wahl \cite{Looijenga-Wahl}, we  give a criterion of the $\lci$ smoothing lifting  for a cusp singularity.  From \cite{Laufer}, \cite{Wahl}, the local embedded dimension of $(X,0)$ is given by
$$\mbox{emb}\dim (X,0)=\max(3, -K_{\widetilde{X}}\cdot K_{\widetilde{X}})$$
where $-K_{\widetilde{X}}\cdot K_{\widetilde{X}}=\sum_{i=1}^{r}(d_i-2)$ if $r\ge 2$ and $-K_{\widetilde{X}}\cdot K_{\widetilde{X}}=d$ if $r=1$. 
Wahl \cite{Wahl} proved that if a cusp singularity $(X,0)$ admits a smoothing, then $\sum_{i=1}^{r}(d_i-3)\le 9$ if $r\ge 2$ and $d\le 10$ if $r=1$.

A result \cite[Theorem 6.6.1]{Looijenga-Wahl} proved the following.  Let $f: (\XX,0)\to (\Delta,0)$ be a smoothing of a cusp singularity $(X,0)$, and let $M$ its Milnor fiber.  Then $\pi_1(M)$ is cyclic and if the local embedded dimension 
of $(X,0)$ is less than $7$, then $\pi_1(M)$ is trivial; and if local embedded dimension 
of $(X,0)$ is  $8$, then $\pi_1(M)$ is trivial or $\zz_2$.

\begin{thm}\label{thm_cusp_singularity_intro}
Let $(X,0)$ be a cusp surface singularity,  and $(\widetilde{X}, E)$ be  its minimal resolution.  Let $f: (\XX,0)\to (\Delta,0)$ be a smoothing of $(X,0)$, and let $G=\pi_1(M)$ be the fundamental group of the Milnor fiber $M$.   Assume that there exists  a $G$-cover $(Y,0)\to (X,0)$ of  $(X,0)$ which  is $\lci$, 
then $(X,0)$ admits an $\lci$ smoothing lifting.
\end{thm}

The above theorems exactly studied the case $(1)$ for log canonical surface singularities.  For the case $(2)$ with local index bigger than one,  they are all rational singularities. The link  $L$ of such singularities is a rational homology sphere.  Thus, 
$H_1(L)_{\tor}$ is a finite abelian group.  From \cite{NW}, \cite{NW2}, there exist universal abelian covers 
of such singularities with transformation group $H_1(L)_{\tor}$ and  the universal abelian cover is $\lci$.  If the universal abelian cover admits an equivariant  smoothing whose quotient gives a smoothing of such a singularity, 
then  this log canonical singularity admits an $\lci$ smoothing lifting. We have

\begin{thm}\label{thm_cusp-quotients_singularity_intro}(Theorem \ref{thm_cusp-quotients_singularity_intro})
Let $(X,0)$ be a log canonical surface singularity of index $>1$.  Let $f: (\XX,0)\to (\Delta,0)$ be a smoothing of $(X,0)$, and let $G=\pi_1(M)$ be the fundamental group of the Milnor fibre $M$.   Assume that $G=H_1(L_X)$, then $(X,0)$ admits an $\lci$ smoothing lifting.
\end{thm}

Note that the condition $\pi_1(M)=H_1(L)$ is, in some sense, too strong,   for example  from \cite{Simonetti}, it is rarely true for $\zz_2$ quotients of cusps.

\subsection{Moduli stack of lci covers}
Let $X$ be a normal slc surface.  From the classification of singularities,  there are only finite set of singular points on $X$ having local index $>1$.   Let $\XX\to X$ be the index one cover DM stack over $X$ defined in \cite{Kollar-Shepherd-Barron}, and reviewed in \cite{Jiang_2022}.  
Then $\XX$ only has $\lci$ singularities, simple elliptic singularities, and cusps.   \cite{Jiang_2021} showed that simple elliptic singularities and cusps with embedded dimension $\ge 6$ have higher obstruction spaces when deforming the slc surfaces. 
Then around such a singularity germ $(X,0)$, there exists an $\lci$ cover $(\widetilde{X},0)\to (X,0)$.  Locally we have a DM stack $[\widetilde{X}/G]$, and we glue all the local constructions to get a 
DM stack $\XX^{\lci}$ over  $X$.   We call it the $\lci$ covering DM stack with coarse moduli space $X$.  The DM stack $\XX^{\lci}$ only has $\lci$ singularities. 

In \cite{Jiang_2022}, the author defined the moduli stack $M^{\lci}$ of $\lci$ cover DM stacks over the moduli stack of slc surfaces.   The moduli stack $M^{\lci}$ is defined as the functor from the category of schemes over $\mathbf{k}$ to the groupoids of isomorphism classes of flat families $\{\XX^{\lci}\to T\}$ of $\lci$ covering DM stacks with the coarse moduli space the $\qq$-Gorenstein families  $\{\sX\to T\}$ of  slc surfaces.   The author  constructed a perfect obstruction theory on the moduli stack $M^{\lci}$. 
Theorem \ref{thm_elliptic_singularity_intro} and Theorem \ref{thm_cusp_singularity_intro} make the construction of the moduli stack of $\lci$ covers in \cite{Jiang_2022} more rigorous.   An advantage to work on the moduli space of $\lci$ covers of slc surfaces is that 
this moduli space admits a virtual fundamental class. 
Theorem  \ref{thm_elliptic_singularity_intro} and Theorem \ref{thm_cusp_singularity_intro} also implies the following corollary:

\begin{cor}\label{cor_Mlci_NPOT}
Let $M=\overline{M}_{K^2, \chi}$ be a KSBA moduli space of slc surfaces with invariants $K_X^2=K^2, \chi(\sO_X)=\chi$.  Suppose that there is a component in $M$ containing slc surfaces $X$ with simple elliptic singularities $(X,0)$ of degrees $6, 7$,  or with cusp singularities $(X,0)$ such that there is no $\lci$ smoothing liftings, then such a component \textbf{cannot} admit a perfect obstruction theory.
\end{cor}
\begin{proof}
Theorem  \ref{thm_elliptic_singularity_intro}  implies that a  simple elliptic singularity $(X,0)$ admits a $\lci$ smoothing lifting if and only if 
$d=1, 2, 3, 4, 8, 9$.   For $d=5$, the singularity germ $(X,0)$, although is not $\lci$,  but its higher obstruction spaces $\sT_{(X,0)}^{q} (q\ge 3)$ all vanish, see \cite[Theorem 1.3]{Jiang_2021}.   For the  simple elliptic singularities $(X,0)$ of degrees $6, 7$, and the  cusp singularities $(X,0)$ such that there is no $\lci$ smoothing liftings (which implies that their local embedded dimension $\ge 6$),  from \cite[Theorem 1.3]{Jiang_2021}, their higher obstruction spaces do not vanish.  So there is no perfect obstruction theory on the moduli space. 
\end{proof}

\subsection{Outline}
The paper is outlined as follows.  In \S \ref{sec_quadratic_function} we review the basic knowledge of quadratic functions on finitely generated abelian groups.  
\S \ref{sec_topology_invariants} defines the topological invariants for a smoothing of surface singularity. 
In 
\S \ref{sec_equivariant_smoothing}, we study the equivariant smoothing of surface singularities. 
Finally in 
\S \ref{sec_log_canonical} we apply the equivariant smoothing to simple elliptic singularities, cusp singularities and cyclic quotients of simple elliptic and cusp singularities.

\subsection*{Convention}

We work over $\cc$.  The link of the surface singularity and the Milnor fibre of a smoothing will be real topology spaces over $\rr$. 
In this paper the homology group $H_i(M):=H_i(M,\zz)$ and cohomology  group $H^i(M):=H^i(M,\zz)$ for a topology space $M$ all take integer coefficient.  
Let $N$ be an abelian group, and $Q_N$ a quadratic function on it.  We denote by 
$\overline{N}:=N/ (\rad(Q_N)\oplus N_{\tor})$ and $Q_N$ the induced quadratic function on $\overline{N}$.

\subsection*{Acknowledgments}

Y. J.  thanks Professor J\'anos Koll\'ar for sending examples of  smoothing of simple elliptic singularity (cone over the degree $6, 7$ del Pezzo surfaces) which do not have  lci  smoothing liftings, and Professor Valery Alexeev for discussions on surface singularities.  
Y. J. thanks Sir Simon Donaldson for the valuable discussions of the degeneration of surface singularities at Simons Center Stony Brook in March 2023. 
Y. J.  is partially supported by Simons Collaboration Grant for Mathematicians  and KU Research GO grant at University of Kansas.


\section{Quadratic functions}\label{sec_quadratic_function}

We recall some basic properties of quadratic functions on abelian groups. 
Let $N$ and $H$ be two abelian groups (i.e.,  modules over $\zz$).  A map $q: N\to H$ is called an $H$-valued quadratic function if the map 
$$b: N\times N\to H$$
defined by $b(x,y)=q(x+y)-q(x)-q(y)$ is bilinear.  This implies that  the function 
$k: N\to H$ given by $k(x):=2q(x)-b(x,x)(=4q(x)-q(2x))$ is a homomorphism.  For any $a\in \zz$, 
$$q(ax)=a^2 q(x)-\frac{1}{2}a(a-1)k(x).$$
We call $k$ (resp. $b$) the linear (resp. bilinear) form associated to $q$. 
If $k=0$, $q$ is called an $H$-valued {\em quadratic form} on $N$.
A {\em quadratic morphism} between $q_1: N_1\to H_1$ and $q_2: N_2\to H_2$ is given by the group homomorphisms
$\phi: N_1\to N_2$ and $\psi: H_1\to H_2$ such that $\phi\circ q_2=q_1\circ \psi$. 

Let $b: N\times N\to H$ be a symmetric bilinear map.  Let $b^\prime: N\to H$ be the adjoint homomorphism defined by 
$b^\prime(x)(y)=b(x,y)$.  The kernel $\ker(b^\prime)$ is called the radical $\rad(b)$ of $b$.  We call $b$ {\em non-degenerate} if $\rad(b)=0$, and
{\em nonsingular} if $b^\prime$ is an isomorphism.  In the classical case that $N$ is free abelian with finite rank, $H=\zz$, then nonsingular is called unimodular.  
Also $b$ induces a nondegenerate symmetric bilinear $\overline{b}: \overline{N}\times \overline{N}\to H$ where 
$\overline{N}=N/\rad(b)$. 
Let $q: N\to H$ be a quadratic function with the associated bilinear form $b$.  We call that $q$ is nondegenerate (resp. nonsingular) if $b$ is. 
The quadratic function $q$ is called {\em ordinary} if the restriction $q|_{\rad(b)}=0$.  Thus, such a $q$ factors through a quadratic function 
$\overline{q}: \overline{N}\to H$ with the associated bilinear form $\overline{b}$.  
Suppose that $I\subset N$ is a subgroup on which $q$ vanishes, we call $I$ to be $q$-isotropic, then $b$ vanishes on $I\times I$. 
Let 
$$I^{\perp}:=\{x\in N| b(x,y)=0 \text{~for all ~} y\in I\}.$$
Then $I^{\perp}\subset N$ is a subgroup and $I\subset I$.  The quadratic function $q$ induces a quadratic function 
$q_I: I^{\perp}/I\to H$. If $q$ is nonsingular, so is $q_I$.

From \cite[\S 1.4]{Looijenga-Wahl}, let $N$ be a free abelian group and $B: N\times N\to \zz$ a symmetric bilinear form. We call $(N, B)$ a bilinear lattice. 
An element $k\in N$ is said to be characteristic for $B$ if $B(x,x)+B(x,k)$ is even for all $x\in N$.  In general, let $K: N\to \zz$ be a linear form such that 
$B(x,x)+K(x)$ is even for all $x\in N$, then $Q(x)=\frac{1}{2}(B(x,x)+K(x))$ is a quadratic function having $K$ (resp. $B$) as the associated linear (resp. bilinear) form. 
The pair $(N, Q)$ is called a quadratic lattice. 

Let us recall the {\em discriminant quadratic function} (DQF) of $(N, Q)$.   Since $b$ is nondegenerate, we identify 
$\Hom(N,\zz)$ with 
$$N^{\#}:=\{x\in N_{\qq}| b(x,y)\in \zz \text{~for all~}y\in N\}.$$
The lattice $N$ is identified with its image in $N_{\qq}$, then $N$ is of finite index in $N^{\#}$. 
Actually $|N^{\#}/N|$ is the absolute value of the discriminant of $B$.  Since $B$ and $K$ extend to $N^{\#}$, so does $Q$. 
For any $x\in N^{\#}$, $y\in N$, 
$$Q(x+y)=Q(x)+Q(y)+B(x,y)\equiv Q(x)(\mod \zz),$$
so $Q$ determines a quadratic function 
$q: N^{\#}/N\to \qq/\zz$.  Accordingly, $B$ induces a bilinear form $b: N^{\#}/N\times N^{\#}/N\to \qq/\zz$ associated to $q$. 
We call $(N^{\#}/N, q)$ the  {\em discriminant quadratic function} (DQF)  of $(N,Q)$. In the case $B$ is even for all $x\in N$, and 
$Q(x)=\frac{1}{2}B(x,x)$, then $(N^{\#}/N, q)$ is the discriminant form of the even lattice $(N,B)$; see \cite[\S 1.3]{Nikulin}. 

More properties of quadratic functions and quadratic forms can be found in \cite[\S 1]{Looijenga-Wahl}, \cite{Nikulin}, \cite{Wall} and the references therein. 
For instance,  
from \cite{Nikulin}, any lattice $N_1$ between $N$ and $N_{\qq}$ on which $Q$ is still integral projects to a $q$-isotropic subspace $I:=N_1/N$ of 
$N^{\#}/N$ with $I^{\perp}=N_1^{\#}/N$, and $q_I: I^{\perp}/I\to \qq/\zz$ may be identified with the DQF of $(N_1, Q)$. Conversely, any  $q$-isotropic subgroup
$I$ of $N^{\#}/N$ comes from a lattice $N_1\subset N$ on which $Q$ is integral.
From \cite{Wall}, every quadratic function $q: G\to \qq/\zz$ with $G$ a finite abelian group, is the DQF of a non-degenerate quadratic lattice. 

In the end of this section we review the linking form of a compact closed oriented $(2n-1)$-manifold $L$.   
We use subscript $\tor$ to present the ``torsion part" of the abelian groups. 
There is a canonical $(-1)^n$-symmetric bilinear map
$$l: H_{n-1}(L)_{\tor}\times H_{n-1}(L)_{\tor}\to \qq/\zz$$
called the linking form of $L$. It is characterized by the property that its adjoint is the composition 
$PD_{\tor}: H_{n-1}(L)_{\tor}\stackrel{\cong}{\rightarrow} H^n(L)_{\tor}$ (torsion part of the Poincare duality) and the isomorphism 
$H^n(L)_{\tor}\cong \Hom(H_{n-1}(L)_{\tor}, \qq/\zz)$ by the universal coefficient formula. 
Therefore, $l$ is non-singular.  The geometric definition is as follows.  Let $v_1, v_2\in H_{n-1}(L)_{\tor}$.
We represent $v_1, v_2$ by disjoint $(n-1)$-cycles $V_1, V_2$.  Then some integral multiple $\lambda V_1$ bounds an $n$-chain $C_1$ and we have
$$l(v_1, v_2)\equiv \frac{1}{\lambda}C_1\cdot V_2 (\mod \zz).$$ 

Suppose that $L$ bounds an oriented compact $2n$-manifold $M$, then we have the composite
$$H_n(M)\stackrel{j}{\longrightarrow}H_n(M,\partial M)\stackrel{LD}{\longrightarrow} H^n(M)\rightarrow \Hom(H_n(M), \zz)$$
is the adjoint of a $(-1)^n$-symmetric bilinear form on $H_n(M)$, the intersection product $(\cdot)$, where $LD$ stands for Lefschetz duality which is an isomorphism. 
Let 
$$\overline{H}_n(M)=H_n(M)/(\mbox{radical}+\mbox{torsion}).$$
Then $(\cdot)$ factors through a non-degenerate bilinear form on $\overline{H}_n(M)$.  We denote by 
$(G_M, b_M)$ the associated discriminant bilinear form.   On the torsion part the isomorphisms 
$LD_{\tor}: H_n(M, L)_{\tor}\stackrel{\cong}{\longrightarrow} H^n(M)_{\tor}$, \quad $H^n(M)_{\tor}\stackrel{\cong}{\longrightarrow}\Hom(H_{n-1}(M)_{\tor}, \zz)$ determines a perfect pairing 
$\alpha: H_n(M,L)_{\tor}\times H_{n-1}(M)_{\tor}\to \qq/\zz.$

\begin{prop}\label{prop_lem_2.4_LW}(\cite[Lemma 2.4]{Looijenga-Wahl})
With respect to the pairing $a$ and $l$ the sequence 
$$H_n(M,L)_{\tor}\stackrel{\partial_{\tor}}{\longrightarrow} H_{n-1}(L)_{\tor}\stackrel{i_{\tor}}{\longrightarrow} H_{n-1}(M)_{\tor}$$
is self-$(-1)^n$-dual: 
$$l(\partial_{\tor}(w),v)=(-1)^n \alpha(w, i_{\tor}(v)).$$
Moreover, $\Im(\partial_{\tor})$ is $l$-isotropic, and 
$\Im(\partial_{\tor})^{\perp}=\ker(i_{\tor})$, and the induced $(-1)^n$-symmetric bilinear form $l_M$ on $\Im(\partial_{\tor})^{\perp}/\Im(\partial_{\tor})$ is canonically isomorphic to 
$(G_M, -b_M)$. 
\end{prop}

\section{Topological invariants of  smoothing surface singularities}\label{sec_topology_invariants}

\subsection{Surface singularity}

Let $(X,0)$ be an isolated surface singularity.  We always assume that $X$ is a Stein space of dimension two with a $\cc^{\infty}$-boundary $\partial X\subset X$.
The space $X$ is topologically the cone over $\partial X$, and $X$ is irreducible and reduced.  We call $\partial X$ the link of $(X,0)$ and denote it by $L_X$. 
Let $\pi: \widetilde{X}\to X$ be a good resolution.  Then $\widetilde{X}$ is a smooth complex analytic surface (with boundary), $\pi$ is proper and the exceptional curve 
$E=\pi^{-1}(0)$ has only normal crossing singularities whose irreducible components $E_1, \cdots, E_r$ are non-singular.  
We give notations for the resolution. Let 
$$g=\sum_i g(E_i);\quad m=\text{the number of loops in the dual graph of $E$}.$$
The space $\widetilde{X}$ retracts topologically onto $E$, therefore, has homology groups 
$H_0(\widetilde{X}), H_1(\widetilde{X}),H_2(\widetilde{X})$ free of rank $1, 2g+m, r$ respectively. 
Then the boundary of $\widetilde{X}$ is also homoemoephic to the link $L_X$.  For the pair $(\widetilde{X}, L_X)$, we have the homology exact sequence:
\begin{equation}\label{eqn_exact_1}
H_2(L_X)\rightarrow H_2(\widetilde{X})\stackrel{j}{\rightarrow} H_2(\widetilde{X}, L_X)\stackrel{\partial}{\rightarrow} H_1(L_X)
\stackrel{}{\rightarrow} H_1(\widetilde{X})\stackrel{}{\rightarrow} H_1(\widetilde{X}, L_X)\to
\end{equation}
Since $H_2(\widetilde{X})\cong H_2(E)$, $H_2(\widetilde{X}, L_X)\cong H_2(E)^*$, $H_1(\widetilde{X})\cong H_1(E)$
and $H_1(\widetilde{X})\cong H^3(E)=0$, the map $j$ is the adjoint of the intersection product.  The group $H_2(\widetilde{X})$ has a basis consisting of the fundamental classes 
$[E_1], \cdots, [E_r]$.  Let $[E_1]^*, \cdots, [E_r]^*$ be the dual basis, then $j$ maps $[E_\alpha]$
to $\sum_{\beta}(E_{\alpha}\cdot E_{\beta})[E_{\beta}]^*$. Since $[E_\alpha\cdot E_{\beta}]$ is negative definite, $\cok(j)$ is torsion and thus we have the exact sequence:
\begin{equation}\label{eqn_exact_torsion_L}
0\to H_1(L_X)_{\tor}\rightarrow H_1(L_X)\rightarrow H_1(\widetilde{X})\to 0.
\end{equation}

Let $T_{\widetilde{X}}$ be the tangent bundle of $\widetilde{X}$ and let 
$K:=K_{\widetilde{X}}\in  H_2(\widetilde{X})^*$ represent the class $-c_1(T_{\widetilde{X}})\in H^2(\widetilde{X})$. We can calculate $K$ by solving the equations 
$K\cdot E_i+E_i\cdot E_i=2 g(E_i)-2$. 
From \cite{Looijenga-Wahl}, the function 
$Q_{\widetilde{X}}(D)=\frac{1}{2}(D\cdot D+ K\cdot D)$ is a quadratic function and $(H_2(\widetilde{X}), Q_{\widetilde{X}})$ is a quadratic lattice. 
Via Poincar\'e duality, the map $H^2(\widetilde{X})\to H^2(L_X)$ is identified with the differential in (\ref{eqn_exact_1}) which implies that its image consists of torsion elements. 
Thus, the first Chern class of the complex structure on $T_{L_X}\oplus \rr_{L_{X}}$ is torsion. 
\cite[Theorem 3.7]{Looijenga-Wahl} showed that the linking form 
$l$ on $H_1(L_X)_{\tor}$ lifts naturally to a quadratic function.  We take $q$ to be the function induced from $Q_{\widetilde{X}}$,  not $-Q_{\widetilde{X}}$.

\subsection{Smoothing}\label{subsec_surface_smoothing}

Let $f: (\XX,0)\to (\Delta,0)$ be a smoothing of the surface singularity $(X,0)$. Here a smoothing means that $f$ is a proper flat analytic morphism and 
$(f^{-1}(0), 0)\cong (X,0)$.   $\XX$ is a contractible Stein space of dimension $3$ with a partial $\cc^{\infty}$-boundary $\partial \XX\subset \XX$ such that 
$f|_{\mbox{int}(\XX\setminus\{0\})}$ and $f|_{\partial \XX}$ are both submersions. 

Over $\Delta\setminus \{0\}$, $f$ is a fibre bundle whose fibres are $2$-dimensional Stein manifolds with boundary. 
Let $M$ denote a typical non-singular fibre called a ``Milnor fibre" of the smoothing.  Also $f|_{\partial \XX}$ is a $\cc^{\infty}$-fibre bundle over $\Delta$.  
This bundle is trivial, i.e., $\Delta\times L_X\to \Delta$ since $\Delta$ is contractible.   Thus, $\partial M=L_X$.  
The Milnor fibre $M$ has the homotopy type of a finite cell-complex of $\dim \le 2$. 
More details can be found in \cite{Milnor}. 
Thus, we have
$$H^i(M)=H_i(M)=0 \text{~for~}i>2;\quad   H_2(M) \text{~is torsion free}.$$
The symmetric intersection product on $H_2(M)$ can be diagonalized over $\rr$ to decompose $\mu=\rk(H_2(M))$ as 
\begin{equation}\label{eqn_mu}
\mu=\mu_0+\mu_{+}+\mu_{-}.
\end{equation}
Here $\mu_0$ is the rank of the radical.  For the pair $(M, \partial M=L_X)$, we have the following exact sequence
\begin{equation}\label{eqn_diagram_M_L}
\xymatrix{
H_2(M)\ar[r]& H_2(M, \partial M)\ar[r]^{\partial}\ar[d]^{\cong}& H_1(\partial M)\ar[r]^{i}\ar[d]^{\cong}& H_1(M)\ar[r]& H_1(M, \partial M)\ar[d]^{\cong}\\
& H^2(M) & H_1(L_X)& &H^3(M)=0
}
\end{equation}
Then we deduce that 
$$\mu_0=b_1(L_X)-b_1(M)=b_1(\widetilde{X})-b_1(M).$$
The complex structure of the fibres of $f$ determines a complex structure on $T_{\XX/\Delta}|_{\partial \XX}=T_{\partial\XX/\Delta}\oplus \rr_{\partial\XX}$. 
From the contraction $\partial\XX\to \partial \XX_0\to L_X$ which trivializes $f: \partial \XX\to \Delta$, there is a continuous family of complex structures on 
$T_{L_X}\oplus \rr_{L_X}$.  The diffeomorphism $\partial M\to L_X$ respects the homotopy classes of the complex structures on $T_{\partial M}\oplus \rr_{\partial M}$ and $T_{L_X}\oplus \rr_{L_X}$;
thus, respects their Chern classes. 
Let $K_M\in \Hom(H_2(M), \zz)$ denote the image of $-c_1(T_{M})$, then 
$$Q_M(x):=\frac{1}{2}(x\cdot x+ \langle K_M, x\rangle)$$
gives $H_2(M)$ the structure of an ordinary quadratic lattice. 
Thus, $Q_M$ induces a quadratic function on $\overline{H}_2(M)$ with DQF $(G_M, q_M)$. 

\begin{thm}\label{thm_Looijenga-Wahl}(\cite[Theorem 4.5]{Looijenga-Wahl})
Suppose that $M$ is the Milnor fibre of a smoothing of an isolated surface singularity $(X,0)$ with link $L_X$. Then the quadratic function $Q_M$ on $H_2(M)$ makes it an ordinary quadratic lattice 
with the associated non-degenerate lattice $(\overline{H}_2(M), Q_M)$.  The corresponding Discriminant Quadratic Form (DQF) is canonically isomorphic to 
$(I^{\perp}/I, q_I)$ where $I\subset H_1(L_X)_{\tor}$ is the $q$-isotropic subgroup $I:=\Im(\partial_{\tor}: H_2(M, L_X)_{\tor}\to H_1(L_X)_{\tor})$.  

Moreover, there exists an exact sequence 
$$0\to H_1(L_X)_{\tor}/I^{\perp}\rightarrow H_1(M)\rightarrow P\to 0$$
where $P$ is a quotient of $H_1(L_X)/H_1(L_X)_{\tor}\cong H_1(\widetilde{X})$. 
\end{thm}

\subsection{Gorenstein surface smoothing}\label{subsec_Gorenstein_smoothing}

A surface singularity $(X,0)$ is {\em Gorenstein} if and only if it is normal and there exists a nowhere zero holomorphic $2$-form $\omega$ on $X\setminus \{0\}$. 

Here is a result of Looijenga-Wahl \cite[Proposition 4.8]{Looijenga-Wahl}:
\begin{prop}
For a Gorenstein surface singularity $(X,0)$, 
\begin{enumerate}
\item $K_{\widetilde{X}}\in H_2(\widetilde{X})$ and $q: H_1(L_X)_{\tor}\to \qq/\zz$ is a quadratic form. 
\item Let $f: (\XX,0)\to (\Delta,0)$ be a smoothing of $(X,0)$ and let $M$ be its Milnor fibre. Then 
$K_M=0$, and $\overline{H}_2(M)$ is an even lattice whose DQF is canonically isomorphic to $(I^{\perp}/I, q_I)$, where 
$I\subset H_1(L_X)_{\tor}$ is a $q$-isotropic subgroup. 
\end{enumerate}
\end{prop}

Let  $f: (\XX,0)\to (\Delta,0)$ be a smoothing of an isolated surface singularity (not necessarily normal).  We already have the topological fact $\mu_0=b_1(\widetilde{X})-b_1(M)$. 
The {\em genus} $p(X)$ of $X$ is defined as
$$p(X)=h^1(\sO_{\widetilde{X}})-\delta(X)$$
where $\delta(X)=\dim_{\cc}(H^0(\sO_{\overline{X}})/H^0(\sO_X))$ which is the colength of the normalization $\overline{X}\to X$. 
From \cite{Steenbrink}, \cite{Greuel-Steenbrink}, we have
\begin{enumerate}
\item  $\mu_0+\mu_{+}=2 p(X)$.
\item If $X$ is normal, (i.e., $\delta(X)=0$), then $b_1(M)=0$. 
\end{enumerate}

The above results put restrictions on the smoothability of the surface singularity $(X,0)$. Let $(X,0)$ be normal, then we get 
$$\mu_0=b_1(\widetilde{X}); \quad  \mu_{+}=2p(X)-b_1(\widetilde{X}).$$
Note that this implies that $\mu_0, \mu_{+}$ only depend on $X$ and not on the smoothing. 
We see below that $\mu_{-}$ depends on the smoothing of $X$. 

Recall that we have $c_1(T_{M}|_{\partial M})=c_1(T_{\widetilde{X}}|_{L_X})$ which is a torsion class. 
Then from the calculation in  \cite{Looijenga3}, we have
$$\frac{1}{3}(K_M\cdot K_M-2\chi(M))-\sign(M)=\frac{1}{3}(K_{\widetilde{X}}\cdot K_{\widetilde{X}}-2\chi(\widetilde{X}))-\sign(\widetilde{X}).$$
Here $\sign(-)$ is the signature.  We can substitute $\mu_0=b_1(\widetilde{X})-b_1(M)$, $\mu_0+\mu_{+}=2p(X)$ and use the fact that 
$(E_i\cdot E_j)$ is negative definite (meaning $\sign(\widetilde{X})=-r$), we get the following formula
\begin{equation}\label{eqn_smoothing_formula}
K_M\cdot K_M+\chi(M)=K_{\widetilde{X}}\cdot K_{\widetilde{X}}+\chi(\widetilde{X})+12 p(X).
\end{equation}

Now suppose that $X$ is Gorenstein, then $K_M=0$ and 
$b_1(M)=0$.  The left hand side of (\ref{eqn_smoothing_formula}) reduces to $1+b_2(M)=1-\mu_{-}+2p(X)$.  Then we have 
\begin{equation}\label{eqn_mu-}
\mu_{-}=10 p(X)+ K_{\widetilde{X}}\cdot K_{\widetilde{X}}-b_1(\widetilde{X})+b_2(\widetilde{X})
\end{equation}
which was proved in \cite{Steenbrink}.

\section{Equivariant smoothing}\label{sec_equivariant_smoothing}

Let $(X,0)$ be an isolated surface singularity. 
We are interested in the smoothing of $(X,0)$ induced by an equivariant smoothing of an $\lci$-cover $(Y,0)\to (X,0)$. 

\subsection{Cover of smoothings}

Let $f: (\XX,0)\to (\Delta,0)$ be a smoothing of $X$, and we denote by $M$ the Milnor fibre of the smoothing $f$. 
Let us first recall a very important result in \cite{Looijenga-Wahl}.

\begin{prop}\label{prop_cover_Milnor_fibre}(\cite[Proposition 5.2]{Looijenga-Wahl})
Any finite unramified covering $\pi: \widetilde{M}\to M$ extends to a finite morphism $\pi: \YY\to \XX$ which is unramified over 
$\XX\setminus \{0\}$. Also $\YY$ is a normal Stein space with partial boundary, and $(\pi^{-1}(0), 0)$ is a surface singularity and 
$f\circ\pi: (\YY,0)\to (\Delta,0)$ is a smoothing of $(Y,0)$ with $\widetilde{M}$ the Milnor fibre of this smoothing. 
\end{prop}
\begin{proof}
This is a very interesting result, and we provide the proof here. 
\cite[Lemma 5.1]{Looijenga-Wahl} showed that $M\subset \XX\setminus\{0\}$ induces an isomorphism on the fundamental group $\pi_1$. 
Thus, $\pi: \widetilde{M}\to M$ extends to an unramified covering $\pi^\prime: U\to \XX\setminus\{0\}$. 
We construct $\pi: \YY\to \XX$ using the direct image sheaf $\pi^\prime_*(\sO_U)$ for $\pi^\prime: U\to \XX\setminus\{0\}\subset \XX$
(by taking $\pi_*(\sO_{\YY})=\pi^\prime_*(\sO_U)$).  Then $\YY$ is a normal contractible space  and a contraction of $\XX$ to $\{0\}$ lifts to a deformation
retraction of $\YY$ onto $\pi^{-1}(0)$, which can only contain a single point $\{0\}$. 
$f\circ\pi: (\YY,0)\to (\Delta,0)$ defines a reduced Cartier divisor $Y\subset \YY$.  Also 
$\partial M\subset \XX\setminus \{0\}$  implies that $X\setminus\{0\}\subset  \XX\setminus \{0\}$. So it induces 
$\pi_1(X\setminus\{0\})\to \pi_1(\XX\setminus\{0\})$ is surjective. Thus, $Y\setminus\{0\}$ is smooth and  connected.   
Therefore, $(Y,0)$ is a surface singularity.
\end{proof}

Let $d=|G|$ be the degree of the unramified cover $\pi: \widetilde{M}\to M$ where $G$ is the finite transformation group.  Then we have two morphisms:
$$\pi^{!}: H_i(M)\to H_i(\widetilde{M}); \quad   \pi_*: H_i(\widetilde{M})\to H_i(M),$$
where $\pi^{!}$ is defined by sending a cycle on $M$ to its pre-image in $\widetilde{M}$. 
These maps satisfy the conditions
$$\pi_{*}\pi^{!}(x)=d x$$
and for $i=2$, $\pi^{!}(x)\cdot \pi^{!}(y)=d (x\cdot y)$.  Below are the inequalities for the covering  $\pi: \widetilde{M}\to M$, see \cite[\S 5.3]{Looijenga-Wahl}.
\begin{enumerate}
\item  $\mu_{i}(\widetilde{M})\ge \mu_i(M) (i=0, +, -)$.
\item $K_{\widetilde{M}}\cdot K_{\widetilde{M}}=d K_{M}\cdot K_M$,  $K_M=0$ if and only if $K_{\widetilde{M}}=0$.
\item $\chi(\widetilde{M})=d\chi(M)$.
\item $p(Y)\ge p(X)$.
\end{enumerate}

We combine the formula (\ref{eqn_smoothing_formula}) and obtain the following formula:

\begin{equation}\label{eqn_cover_formula}
d(K_{\widetilde{X}}\cdot K_{\widetilde{X}}+\chi(\widetilde{X})+12 p(X))=K_{\widetilde{Y}}\cdot K_{\widetilde{Y}}+\chi(\widetilde{Y})+12 p(Y).
\end{equation}

There is another formula involving the resolution  $\widetilde{Y}$ of $Y$.  We have 
\begin{equation}\label{eqn_formula_Y}
p(X)\le p(Y)\le h^1(\sO_{\widetilde{Y}}).
\end{equation}

\subsection{Proof of Theorem \ref{thm_smoothing_component_lci_intro}}

The result of Theorem \ref{thm_smoothing_component_lci_intro} goes through if we write down the equivariant smoothing carefully. 
Suppose that in the cover smoothing 
$$\tilde{f}=f\circ \pi: (\YY,0)\to (\Delta,0),$$
$(Y,0)$ is $\lci$, which means $(\YY,0)$ is also $\lci$. 
Let $d=|G|$ be the order of the transformation group $G$. 
The link $L_{\YY}$ of the threefold singularity is simply connected, which implies that 
$H_1(\widetilde{M})=0$. 
So 
$$H_2(\widetilde{M}, L_{Y})_{\tor}\cong H^2(\widetilde{M})_{\tor}=\Hom(H_1(\widetilde{M})_{\tor},\qq/\zz)=0.$$
Therefore, we have $\Im(\partial_{\tor}: H_2(\widetilde{M}, L_{Y})_{\tor}\to H_1(L_Y)_{\tor})=0$.  Thus, there is no 
nontrivial $I\subset  H_1(L_Y)_{\tor}$.  From the construction of the unramified cover 
$$\pi: \widetilde{M}\to M$$
which induces an unramified $G$-cover $\pi: L_Y\to L_X$ on the boundaries.  Also the morphisms 
$$\pi^{!}: H_i(M)\to H_i(\widetilde{M})$$
which sends a cycle in $M$ to its pre-image in $\widetilde{M}$ (which is  also a cycle) satisfying 
$$
\pi_*\pi^{!}(x)=d x; \quad  \text{and when~} i=2, \pi^{!}(x)\cdot \pi^{!}(y)=d (x\cdot y).
$$
Similar formula holds for  $\pi: L_Y\to L_X$.
It is routine to check that the morphisms $\pi^{!}: H_2(M)\to H_2(\widetilde{M})$ and 
$\pi_*: H_2(\widetilde{M})\to H_2(M)$  (also for $\pi^{!}: H_2(L_X)\to H_2(L_Y)$ and 
$\pi_*: H_1(L_Y)\to H_1(L_X)$) are compatible with the quadratic functions. 

In this case we have 
$$DQF(H_2(\widetilde{M}), \widetilde{Q})\cong (H_1(L_Y)_{\tor}, \tilde{q}).$$
Then we get a  $4$-tuple $(H_2(\widetilde{M}), \widetilde{Q}, \tilde{i}, \tilde{s})$ 
and an element 
$$(H_2(\widetilde{M}), \widetilde{Q}, \tilde{i}, \tilde{s})\to (H_2(M), I, i, s))$$ 
in 
$\mathscr{S}\left((\widetilde{\mu}_{+},A_Y, \tilde{q})\rightarrow(\mu_+, A_X, q)\right)$.

\section{Log canonical surface singularities}\label{sec_log_canonical}

In this section we focus on log canonical surface singularities-simple elliptic singularities, cusps, and their cyclic quotients.  
Simple elliptic singularities and cusp singularities  are both Gorenstein surface singularities, and their cyclic quotients are $\qq$-Gorenstein singularities. 

\subsection{Gorenstein covers}\label{subsec_Gorenstein_cover}

Let $(X,0)$ be a Gorenstein surface singularity.  Then from \S \ref{subsec_surface_smoothing}, and \S \ref{subsec_Gorenstein_smoothing}, 
we assign the smoothing data for a smoothing component of $(X,0)$. 
From (\ref{eqn_diagram_M_L}), 
the data consist of the even Milnor lattice $H_2(M)$ (where $M$ is the Milnor fibre of the smoothing component);  a $q$-isotropic subgroup
$I\subset H_1(L_X)_{\tor}$;  an injection 
$$i: \cok\left(H_2(M)\to H^2(M)\right)\to H_1(L_X)/I$$
which induces an isomorphism 
$\overline{H}_2(M)^{\#}/\overline{H}_2(M)\stackrel{\sim}{\rightarrow} I^{\perp}/I$; and if $\mu_{+}>0$ a sign structure on 
$H_2(M)$, up to equivalence. 
Also $\cok(i)\cong H_1(M)$. Then together with (\ref{eqn_cover_formula}) and 
(\ref{eqn_formula_Y}), these data put conditions on the covering of $M$, which in turn impose restrictions on $I$. 

Now let  $(X,0)$ be a minimally elliptic singularity \cite{Laufer}, which is, by definition, a Gorenstein singularity having genus $1$.  This property can be 
recognized by its resolution graph.  From \cite[Theorem 3.13]{Laufer},
its local embedded dimension 
$$\mbox{emb} \dim (X,0)=\max(3, -K_{\widetilde{X}}\cdot K_{\widetilde{X}})$$
where $\widetilde{X}$ is the minimal resolution.  Also $(X,0)$ is an $\lci$ singularity if and only if 
$d=-K_{\widetilde{X}}\cdot K_{\widetilde{X}}\le 4$.
In the $\lci$ case, $\pi_1(M)$ is trivial. 

\begin{defn}\label{defn_permissible_quotient}
Let $(X,0)$ be a Gorenstein surface singularity and 
$\phi: H_1(L_X)\to G$ a surjective morphism to a finite group. 
The surjection induces a normalized $G$-covering $(Y,0)\to (X,0)$. 
Let $d=|G|$.
The map $\phi$ is called an {\em impermissible quotient} if for each integer $p$ with $p(X)\le p\le p(Y)$, we have 
\begin{equation}\label{eqn_impermissible_formula}
d\left(\chi(\widetilde{X})+K_{\widetilde{X}}\cdot K_{\widetilde{X}}+12p(X)\right)\neq \chi(\widetilde{Y})+K_{\widetilde{Y}}\cdot K_{\widetilde{Y}}+12p(Y).
\end{equation}

The finite quotient $\phi: H_1(L_X)\to G$ is called {\em permissible} if it is not a factor of an impermissible quotient. 
\end{defn}

From (\ref{eqn_cover_formula}), an impermissible covering can not be induced by a $G$-covering of any Milnor fibre. 
Let $(X,0)$ be minimally elliptic (a simple elliptic or cusp singularity), then $p(X)=p(Y)=1$, then (\ref{eqn_impermissible_formula}) reduces to 
\begin{equation}\label{eqn_impermissible_formula2}
d\left(\chi(\widetilde{X})+K_{\widetilde{X}}\cdot K_{\widetilde{X}}+12\right)\neq \chi(\widetilde{Y})+K_{\widetilde{Y}}\cdot K_{\widetilde{Y}}+12.
\end{equation}

From \cite[\S 6.2]{Looijenga-Wahl}, we have the following two properties:
\begin{enumerate}
\item Every nontrivial quotient $H_1(L_X)\to H_1(L_X)/J$ which kills the torsion $H_1(L_X)_{\tor}$ is impermissible. 
\item Every automorphism of $H_1(L_X)$ which is the identity on both the torsion $H_1(L_X)_{\tor}$ and $H_1(L_X)/H_1(L_X)_{\tor}$
is induced by an automorphism of $X$.
\end{enumerate}

For simple elliptic and cusp singularities,  we explain how the smoothing data in \S ?? of $(X,0)$ is related to a 
$G$-cover $(Y,0)\to (X,0)$.  For such surface singularities, we have $\rk(H_1(L_X))>1$. Then $(1)$ implies that if the nontrivial quotient 
$\phi: H_1(L_X)\to H_1(L_X)/J$  is permissible, we must have
$$H_1(L_X)_{\tor}+J=H_1(L_X).$$
If not, $\phi$ is a factor of the impermissible $H_1(L_X)\to H_1(L_X)/(J+H_1(L_X)_{\tor})$.

Let us fix a $q$-isotropic subgroup $I\subset H_1(L_X)_{\tor}$, and let $\mathscr{J}(I)$ denote the collection of subgroups $J\subset H_1(L_X)$
with $H_1(L_X)_{\tor}\cap J=I^{\perp}$ and $H_1(L_X)_{\tor}+J=H_1(L_X)$.  Let $\Gamma$ be the group of automorphisms of 
$H_1(L_X)$ acting trivially on $H_1(L_X)_{\tor}$, and $H_1(L_X)/H_1(L_X)_{\tor}$. 
Every $g\in \Gamma$ is of the form 
$$x\mapsto x+\alpha(x)$$
for some $\alpha\in \Hom(H_1(L_X)/H_1(L_X)_{\tor}, H_1(L_X)_{\tor})$.   Thus, we identify 
$\Gamma$ with $\Hom(H_1(L_X)/H_1(L_X)_{\tor}, H_1(L_X)_{\tor})$.  The subgroup of $\Gamma$ corresponding to 
$\Hom(H_1(L_X)/H_1(L_X)_{\tor}, I^{\perp})$ acts trivially on $\mathscr{J}(I)$, and 
$\Hom(H_1(L_X)/H_1(L_X)_{\tor}, H_1(L_X)_{\tor}/I^{\perp})\cong \Hom(H_1(L_X)/H_1(L_X)_{\tor}, I)$ acts simply transitively on $\mathscr{J}(I)$. 
We have $|\mathscr{J}(I)|=|I|^{\mu_0}$. 

Condition $(2)$ implies that $\mathscr{J}(I)$ is contained in an $\Aut(X)$-orbit. Thus, permissibility only depends on $I$.  
We call $I$ permissible if sone (equivalently, every) quotient $H_1(L_X)\to H_1(L_X)/J$ with $H_1(L_X)_{\tor}\cap J=I^{\perp}$, is permissible.  
The automorphism group $\Aut(X)$ acts on the set of its smoothing components. 
Let $\mathscr{S}_{\perm}(X)\subset \mathscr{S}(\mu_+, H_1(L_X), q)$ be the subset corresponding to permissible $I$'s. 
Then the map
$$\Phi^{\lci}: |\MM^{\sm}_{\lci}|\to \mathscr{S}_{\perm}(X)\subset \mathscr{S}(\mu_+, H_1(L_X), q)$$ 
is clearly $\Aut(X)$-invariant.

\subsection{Proof of Theorem \ref{thm_elliptic_singularity_intro}}

The result is essentially given in \cite[Example 6.4]{Looijenga-Wahl}, although they didn't mention that the liftings are all $\lci$ covers.  We provide the details here. 

Let $(X,0)$ be a simple elliptic surface singularity and $(\widetilde{X}, E)\to (X,0)$ a good resolution. 
From \cite[Theorem 3.13]{Laufer},  $(X,0)$ is an $\lci$ singularity if the negative self-intersection $-E\cdot E\le 4$.  If $-E\cdot E\ge 5$, then $(X,0)$ is never $\lci$. 
Let $d:=-E\cdot E$ be the degree of $(X,0)$.  From \cite{Pinkham}, a simple elliptic singularity $(X,0)$ admits a smoothing if and only if $1\leq d\leq 9$.

We have $\pi_1(L_X)=\zz\rtimes \zz^2$ and $H_1(L_X)=\zz^2\oplus \zz_d$ for $(X,0)$, since $H_2(\widetilde{X})^{\#}/H_2(\widetilde{X})\cong \zz_d$.  
Then the class $\frac{1}{d}[E]$ in $H^2(\widetilde{X})$ generates the cyclic group $\zz_d$, so does its image $e\in H_1(L_X)_{\tor}$.  
The quadratic form
$q$ on $H_1(L_X)_{\tor}$ is given by $q(e)=\frac{d-1}{2d}\in\qq/\zz$.   Also from (\ref{eqn_mu-}), 
$$\mu_{+}=0;\quad  \mu_-=9-d$$
for $1\le d\le 9$.

We already know that $(X,0)$ is an $\lci$ singularity for $1\le d\le 4$.  Therefore, we only need to check $d=5, 6, 7, 8, 9$. 
Recall that a $q$-isotropic subgroup $I\subset H_1(L_X)_{\tor}$ is permissible if the quotient map 
$H_1(L_X)\to H_1(L_X)/J$ for $J\subset H_1(L_X)$ is permissible, where $J\cap H_1(L_X)_{\tor}=I^{\perp}$. 
Now we check that for $d=5, 6, 7, 8, 9$, the only non-trivial permissible $I\subset H_1(L_X)_{\tor}=\zz_d$ are 
$I=\zz_2$ when $d=8$; and $I=\zz_3$ when $d=9$.   
This can be easily checked using the formula 
(\ref{eqn_impermissible_formula2}).  Note that for $d=5, 6, 7$, there is no non-trivial $I\subset \zz_d$ satisfying the formula (\ref{eqn_impermissible_formula2}).

When $d=8$,  $I=\zz_2\subset \zz_8$, which is generated by $4e$, and we check 
$$q(4e)=4^2\cdot \frac{7}{16}=0\in \qq/\zz.$$
Then we have that $H_1(L_X)/J=\zz_4$ and we get a $\zz_2$-covering 
$$\pi: (Y,0)\to (X,0)$$
and $(Y,0)$ is a simple elliptic singularity with degree $4$, and hence $\lci$. 
Since the above $\zz_2$-covering is permissible,  it comes form a $\zz_2$ covering 
$$\pi: \widetilde{M}\to M$$
for the Milnor fibre of $M$ of a smoothing $f: (\XX,0)\to (\Delta,0)$.
From Proposition \ref{prop_cover_Milnor_fibre}, there is a smoothing 
$f\circ\pi: (\YY,0)\to (\Delta,0)$ of the simple elliptic singularity
$(Y,0)$ with Milnor fibre $\widetilde{M}$.  The $\zz_2$ acts on $\YY$ with quotient $\XX$.  Thus the $\zz_2$-equivariant smoothing 
of $(Y,0)$
is the smoothing $f: (\XX,0)\to (\Delta,0)$ of $(X,0)$. 

When $d=9$,  $I=\zz_3$, then we check that $H_1(L_X)/J=\zz_3$ and we get a $\zz_3$-covering 
$$\pi: (Y,0)\to (X,0)$$
and $(Y,0)$ is a simple elliptic singularity with degree $3$, and hence also $\lci$ (actually a hypersurface singularity). 
Similar analysis implies that we have a  $\zz_3$-equivariant smoothing 
of $(Y,0)$ which induces 
 the smoothing $f: (\XX,0)\to (\Delta,0)$ of $(X,0)$. 
 
\subsection{Smoothing components of simple elliptic singularities}
 The smoothing components $|\MM^{\sm}|$ of $(X,0)$ can be given by the set of subgroups $J\subset H_1(L_X)$ with 
 $$J\cap H_1(L_X)_{\tor}=I^{\perp}$$ 
 where $I\subset H_1(L_X)_{\tor}$ is 
 $q$-isotropic and $J+H_1(L_X)_{\tor}=H_1(L_X)$.  
 
 Now let  $d=8$, the $q$ isotropic subgroups of $\zz_8$ is $I=\zz_2$.  We can calculate that $I^{\perp}=\zz_4$.  There are totally 
 $1+2^2=5$ such elements.  From \cite{Pinkham},  each permissible $I\subset H_1(L_X)_{\tor}=\zz_8$ is realized by a smoothing component. 
 
 A smoothing of $(X,0)$ is essentially given by a Del Pezzo surface 
 $Z$ of degree $8$, and a non-singular anti-canonical divisor $D$ on $Z$ such that 
 $(D, N_{D/Z})\cong (E, N_{E/\widetilde{X}})$ as polarized curves (\cite[Proposition A.6]{Looijenga4}).
 The automorphic group $\Aut(X)$ acts on these smoothing components, and the  $\Aut(X)$-orbits are given by the 
 coarse moduli space of the pairs $(Z, D)$ with $D$ isomorphic to $E$.   When $d=8$, this coarse moduli space 
 has two connected components.  From \cite[\S 9]{Pinkham},  the germ coarse moduli space is given by 
 $$\spec(R)=V_1\cup_{V_j}V_2,$$
 where $V_j$ is the $j$-line of the elliptic curve, which represents the simple elliptic  singularities $(X,0)$ which are cones over elliptic curves; 
 $V_1\setminus V_j$ represents the rational surfaces $F_1$ which can be embedded by a linear  system $L$ of cubic with one base point  to $\pp^8$; and 
 $V_2\setminus V_j$ represents the rational surfaces $F_2$ which is the Veronese transform of a quadric $\pp^1\times\pp^1$ in $\pp^3$. 
 The fundamental group of the Milnor fiber is trivial in the first case and $\zz_2$ in the latter.
 We can take $F_1$ and $F_2$ as the Milnor fibres of the smoothings.   Then our $\lci$ smoothing component must be the component 
 $V_2$. 
 
When $d=9$, since for the simple elliptic surface singularity $(X,0)$, the homology of the link is 
$H_1(L_X)=\zz^2\oplus\zz_9$, and $q: H_1(L_X)_{\tor}\to \qq/\zz$ is given by $q(e)=\frac{4}{9}$, 
we have for any $3e\in \zz_9=H_1(L_X)_{\tor}$, $q(3e)=3^2\cdot \frac{4}{9}=0$. 
So $I=\zz_3$ is $q$-isotropic.  We can calculate $I^\perp=\zz_3$ too. 
In this case  if we have a smoothing $f: (\XX,0)\to (X,0)$, we have $\overline{H}_2(M)=0$ and $I^{\perp}/I=0$, so $I\neq 0$. 
There are totally $3^2=9$ subgroups $J\subset H_1(L_X)=\zz^2\oplus\zz_9$ with the properties that 
$J\cap H_1(L_X)_{\tor}=I^{\perp}$, where $I\subset H_1(L_X)_{\tor}$ is   $q$-isotropic and $J+H_1(L_X)_{\tor}=H_1(L_X)$.  Each permissible $I\subset H_1(L_X)_{\tor}=\zz_9$ is realized by a smoothing component, and therefore, there are totally $9$ smoothing components.  

Each $I\cong \zz_3$ gives a $\zz_3$-cover of the smoothing 
$$(\YY,0)\to (\XX,0)$$
such that $(Y,0)\to (X,0)$ is a $\zz_3$-cover of a simple elliptic singularity of degree $3$, which is a hypersurface singularity. 
Any smoothing in the $9$ smoothing components  admits a lci smoothing lifting. 

In the case of $d=9$, the above result matches the result in Koll\'ar \cite[\S 2.26]{Kollar-Book}.
We include the result here.  Let $(X,0)$ be the simple elliptic singularity of degree $d=9$ and let $E$ be the smooth elliptic curve in the minimal resolution. 
A degree $9$ line bundle $L_9$ on $E$ determines an embedding $E\hookrightarrow E_9\subseteq\pp^8$. 
The embedding $E\hookrightarrow \pp^2$ is given by a degree $3$ line bundle $L_3$ on $E$. 
We then embed $\pp^2$ into $\pp^9$ by $\sO_{\pp^2}(3)$.  Now $E$
is mapped to $E_9$ if and only if $L_3^{\otimes 3}\cong L_9$. For a fixed line bundle  $L_9$ on $E$,  since the $3$-torsion of the Picard group of $E$ is $\zz_3^{\oplus 2}$,  there are totally  $9$ choices of $L_3$ that $L_3^{\otimes 3}\cong L_9$. Thus a 
given $E_9\hookrightarrow \pp^8$ is a hyperplane section of a $\pp^2\hookrightarrow \pp^9$ in $9$ different ways. Therefore, the deformation space 
$(X,0)$ has $9$ smoothing components.

\subsection{Proof of Theorem \ref{thm_cusp_singularity_intro}}

Now we let $(X,0)$ be a cusp singularity whose resolution cycle $E$ has self-intersection sequence
$[-d_1, \cdots, -d_r]$.  In this case we have 
$\pi_1(L_X)=\zz^2\rtimes\zz$ and $H_1(L_X)=\zz\oplus H_1(L_X)_{\tor}$  for cusp singularities. 
The monodromy of the link is given by the matrix 
$$
A=
\mat{cc}
0&-1\\
1&d_r
\rix\cdots \mat{cc}
0&-1\\
1&d_1
\rix
=
\mat{cc}
a&b\\
c&d
\rix.
$$
Then we have $H_1(L_X)_{\tor}\cong \zz^2/(A-\id)$. 

Let $f: (\XX,0)\to (\Delta,0)$ be a smoothing of $(X,0)$ with  fibre 
$M$. 
We will search the $G$-covers $(\YY,0)$ of $(\XX,0)$ (i.e., unramified $G$-covers $\widetilde{M}\to M$ for the Milnor fibre $M$), such that 
$(\YY,0)$ is a $G$-equivariant smoothing of the $G$-cover $(Y,0)\to (X,0)$ and $(Y,0)$ is also $\lci$.  
In this case 
$-K_{\widetilde{X}}\cdot K_{\widetilde{X}}=\sum_{i=1}^{r}(d_i-2)$ if $r\ge 2$, and $-K_{\widetilde{X}}\cdot K_{\widetilde{X}}=d$ if $r=1$. 
We have 
$$\mu_0=\mu_{+}=1;  \mu_{-}= 9-\sum_{i=1}^{r}(d_i-3) (r\ge 2) \text{~or~}10-d (r=1).$$
Then \cite{Looijenga-Wahl} studied the possible permissible $I\subset H_1(L_X)_{\tor}$ such that there exists a surjection
$$H_1(L_X)\to G$$
to a finite group $G$.  We have the following result from \cite[Theorem 6.6.1]{Looijenga-Wahl}.

The Milnor fibre $M$ of a smoothing  $f: (\XX,0)\to (\Delta,0)$ satisfies the property that 
$\pi_1(M)$ is a cyclic  group; and 
$\mbox{emb}\dim (X,0)\le 7$ (resp. $=8$) implies that 
$\pi_1(M)$ is trivial (resp. of order $\le 2$). 

Since $\pi_1(L_X)\to \pi_1(M)$ is surjective, we have the following diagram:
\begin{equation}\label{eqn_diagram-cusp}
\xymatrix{
& 0\ar[d] & 0\ar[d] &0\ar[d] &\\
0\ar[r]&K^\prime\ar[r]\ar[d]& K\ar[r]\ar[d]& K^{\prime\prime}\ar[r]\ar[d] &0\\
0\ar[r]&K_{\pi}^\prime\ar[r]\ar[d]& \pi_1(L_X) \ar[r]\ar[d]& \pi_1(M)\ar[r]\ar[d] &0\\
0\ar[r]&K_{H}^\prime\ar[r]\ar[d]& H_1(L_X) \ar[r]\ar[d]& H_1(M)\ar[r]\ar[d] &0\\
& 0 & 0 &0 &\\
}
\end{equation}
where the second row of vertical arrows are abelianization maps. 
Since $\pi_1(M)$ is cyclic,  $\pi_{1}(M)\cong H_1(M)$. 

For the cusp singularity $(X,0)$, we recall the method in \cite[Proposition 4.1]{NW} to construct the $\lci$ cover of $(X,0)$. 
Let $H$ be the subspace of $\zz^2$ generated by $\mat{c}a\\ c\rix$ and $\mat{c}0\\ 1\rix$.    Then the  matrix $A$ takes the subspace $H$ to itself
by the matrix  $\mat{cc}0&-1\\ 1& t\rix$ where $t=\tr(A)=a+d$.
The finite transformation group $G$ is  given as follows:   first we take  the quotient finite group $\pi_1(L_X)/H\rtimes \zz$, the subgroup 
$H\rtimes \zz\subset \pi_1(L_X)$ determines a cover  of $X$. 
This cover determined by $H\rtimes \zz$ is either the cusp with resolution graph consisting of a cycle with one vertex weighted $-t$ or the dual cusp of this, according as the above basis is oriented correctly or not, i.e., whether $a< 0$ or  $a > 0$.
By  taking the discriminant cover if necessary we get the cover $(Y, 0)$ of $X$  with transformation group 
$G$. 

At least for $a\ge 0$,  the finite transformation group $G=\zz_t$ for $t=a+d$.  Then the $G$-cover $(Y,0)$ is a cusp with resolution cycle 
(which is the dual cycle of the cycle $[-t]$)
$$[-3, \underbrace{-2,\cdots, -2}_{t-3}].$$
From \cite{Nakamura}, $(Y,0)$ is a hypersurface cusp singularity (hence $\lci$). 

Now suppose that $\pi_1(M)=\zz_t$,  then we modify the  diagram (\ref{eqn_diagram-cusp}) and get 
\begin{equation}\label{eqn_diagram-cusp2}
\xymatrix{
& 0\ar[d] & 0\ar[d] & &\\
0\ar[r]&H\rtimes\zz\ar[r]\ar[d]& H\rtimes\zz\ar[r]\ar[d]& 0\ar[d] &\\
0\ar[r]&H\rtimes \zz\ar[r]\ar[d]& \pi_1(L_X) \ar[r]\ar[d]& \pi_1(M)=\zz_t\ar[r]\ar[d] &0\\
&0\ar[r]& \zz_t \ar[r]\ar[d]& \zz_t\ar[r]\ar[d] &0\\
&  & 0 &0 &\\
}
\end{equation}
Then the $G=\zz_t$-cover of $(X,0)$ gives the hypersurface cusp 
$(Y,0)$ and the $\zz_t$-cover $\widetilde{M}\to M$. 
Thus, from Proposition \ref{prop_cover_Milnor_fibre},  there is a smoothing 
$$f\circ\pi: (\YY,0)\to (\Delta)$$
of $(Y,0)$ such that $\pi: (\YY,0)\to (\XX,0)$ is a $\zz_t$-cover and the $\zz_t$-action on $\YY$ only fixes the point $0$. 
Therefore, $f\circ\pi: (\YY,0)\to (\Delta)$ is a $G$-equivariant smoothing which induces the smoothing $f: (\XX,0)\to (\Delta,0)$
of $(X,0)$.

\subsection{Quotients of simple elliptic and cusp singularities}\label{sec_quotient_cusp}

Recall from the classification of slc singularities in \cite[Theorem 4.24]{Kollar-Shepherd-Barron}, and recalled in \S \ref{subsec_slc_intro}, the log canonical surface singularities with index $>1$ are $\zz_2, \zz_3, \zz_4,\zz_6$-quotients of simple elliptic singularities and $\zz_2$-quotients of cusps.  Their resolution graphs can be found in \cite[Theorem 9.13]{Kawamata}. These are all rational singularities since their resolution graphs are trees of $\pp^1$'s.  These singularities satisfy the conditions in \cite{NW}, \cite{NW2}. Let $(X,0)$ be such a singularity,  then the link $L_X$ is a rational homology sphere. 
The homology $H_1(L_X)$ is a finite abelian group.  There is a surjective map:
$$\pi_1(L_X)\to H_1(L_X)\to 0.$$
Then we take the finite abelian cover 
$$\pi: (Y,0)\to (X,0)$$
with transformation group $G=H_1(L_X)$ which is called the universal abelian cover. 

From similar proof as in Theorem \ref{thm_cusp_singularity_intro} we have 
\begin{thm}\label{thm_cusp-quotients_singularity}
Let $(X,0)$ be a log canonical surface singularity of index $>1$.  Let $f: (\XX,0)\to (\Delta,0)$ be a smoothing of $(X,0)$, and let $G=\pi_1(M)$ be the fundamental group of the Milnor fibre $M$.   Assume that $G=H_1(L_X)$, then $(X,0)$ admits an $\lci$ smoothing lifting.
\end{thm}
\begin{proof}
The proof is similar to Theorem \ref{thm_cusp_singularity_intro}.  Look at the following diagram (similar to Diagram \ref{eqn_diagram-cusp2})
\begin{equation}\label{eqn_diagram-3}
\xymatrix{
& 0\ar[d] & 0\ar[d] & &\\
0\ar[r]&H\rtimes\zz\ar[r]\ar[d]& H\rtimes\zz\ar[r]\ar[d]& 0\ar[d] &\\
0\ar[r]&H\rtimes \zz\ar[r]\ar[d]& \pi_1(L_X) \ar[r]\ar[d]& \pi_1(M)=G\ar[r]\ar[d] &0\\
&0\ar[r]& G \ar[r]\ar[d]& G\ar[r]\ar[d] &0\\
&  & 0 &0 &\\
}
\end{equation}
If we have a smoothing  $f: (\XX,0)\to (\Delta,0)$  of $(X,0)$ such that the Milnor fibre $M$ satisfies $G=\pi_1(M)$, then we can take the $G$-cover of the smoothing to get a smoothing 
$$\tilde{f}: (\YY,0)\to (\Delta,0)$$
such that the singularity $(Y,0)\to (X,0)$ is a $G$-cover.   The cover $(Y,0)$ must be $\lci$ since $H_1(L_X)=G$.   The smoothing $\tilde{f}: (\YY,0)\to (\Delta,0)$ is $G$-equivariant and its quotient gives the smoothing $f: (\XX,0)\to (\Delta,0)$.
\end{proof}




\subsection*{}


\begin{thebibliography}{12}  




\bibitem{Donaldson} S. K. Donaldson, \newblock Fredholm topology and enumerative geometry: reflections on some words of Michael Atiyah, {\em Proceedings of 26th G\"okova Geometry-Topology Conference},  1-31.  

\bibitem{Engel} P. Engel,  \newblock A proof of Looijenga's conjecture via integral-affine geometry, 
{\em Journal of Differential Geometry},  109(3): 467-495, 2018. 

\bibitem{Friedman}R. Friedman,  \newblock Global smoothing of varieties with normal crossings, {\em Annals Math.}, 118 (1983), 75-114.
\bibitem{Friedman2}R. Friedman,  \newblock  On the geometry of anti-canonical pairs, arXiv:1502.02560.
\bibitem{FM}R. Friedman and R. Miranda,  \newblock Smoothing cusp singularities of small length, {\em Math. Ann.} 263, 185-212 (1983).
\bibitem{Goresky-MacPherson} M. Goresky  and  R. MacPherson, \newblock  Stratified Morse Theory
 Ergebnisse der Mathematik und ihrer Grenzgebiete. 3. Folge / A Series of Modern Surveys in Mathematics (MATHE3, volume 14).
\bibitem{GHK15} M. Gross, P. Hacking and S. Keel,  \newblock  Mirror symmetry for log Calabi-Yau surfaces I, {\em Publ. Math. Inst. Hautes Etudes Sci.}, 122:65-168, (2015). 
\bibitem{Greuel-Steenbrink} G.-M. Greuel and J. H. M. Steenbrink, \newblock On the topology of smoothable singularities, {\em Proc. Symp. Pure Math.} 40 Part I (1983), 535-545.
\bibitem{Hamm} H. Hamm, \newblock  Lokale topologische Eigenschaften komplexer R\"aume,  {\em Math. Annalen}, 191:235-252, 1971.
\bibitem{Hirzebruch} F. Hirzebruch, \newblock Hilbert modular surfaces, {\em Ens. Math.} 19 183-281 (1973). 

\bibitem{Inoue}M. Inoue, \newblock New surfaces with no meromorphic functions, II, In {\em Complex analysis and algebraic geometry}, 91-106, Iwanami Shoten, Tokyo, 1977. 
\bibitem{Jiang_2021}Y. Jiang, \newblock A note on higher obstruction spaces for surface singularities, preprint, arXiv:2112.10679. 

\bibitem{Jiang_2022}Y. Jiang, \newblock The virtual fundamental class for the moduli space of surfaces of general type, arXiv:2206.00575.

\bibitem{Jiang_cusp}Y. Jiang, \newblock On equivariant smoothing of cusp singularities, arXiv:2302.00637. 
\bibitem{Kawamata}Y. Kawamata, \newblock  Crepant blowings-up of three dimensional canonical singularities and its application to degenerations of surfaces, {\em Ann. Math.},  Vol. 127, No. 1 (1988), 93-163.

\bibitem{Kollar-Shepherd-Barron} J. Koll\'ar  and N. I. Shepherd-Barron, \newblock Threefolds and deformations of surface singularities, 
{\em Invent. Math.}, 91, 299-338 (1998).

\bibitem{Kollar-Book}J. Koll\'ar, \newblock  Families of Varieties of General Type,  (Cambridge Tracts in Mathematics). Cambridge: Cambridge University Press. 

\bibitem{Laufer} H. Laufer, \newblock On minimally elliptic singularities, {\em Amer. J. Math.}, Vol. 99, No. 6 (1977), 1257-1295.
\bibitem{Le} D. T. L\^e, \newblock   Sur les cycles \'evanouissants des espaces analytiques,  {\em C. R. Acad. Sci. Paris}, S\'er. A-B,
288:A283-A285, 1979.
\bibitem{Looijenga2}E. Looijenga, \newblock  Rational surfaces with an anticanonical cycle, {\em Ann. of Math.} (2), 114 (2):267-322, 1981.
\bibitem{Looijenga}E. Looijenga, \newblock Isolated singular points on complete intersections, London Math. Soc. Lecture Notes Series 77, Cambridge Univ. Press, Cambridge (1984). 
\bibitem{Looijenga3}E. Looijenga, \newblock Riemann-Roch and smoothing of singularities,  {\em Topology}, Vol.25,  No. 3.,  243-302,  (1986).
\bibitem{Looijenga4}E. Looijenga, \newblock  The smoothing components of a triangle singularity, II, {\em Math. Ann.} 269 (1984), 357-387. 
 
\bibitem{Looijenga-Wahl}E. Looijenga and J. Wahl,   \newblock Quadratic functions and smoothing surface singularities, {\em Topology}, Vol. 25, No. 3, 261-291 (1986).
\bibitem{Milnor}J. Milnor, \newblock {\em Singular points of complex hypersurfaces}, Ann. of Math. Studies 91, Princeton Univ. Press, Princeton (1968).
\bibitem{Nakamura} I. Nakamura, \newblock  Inoue-Hirzebruch surfaces and a duality of hyperbolic unimodular singularities. I. {\em Math. Ann.} 252, 221-235 (1980). 
\bibitem{NW}W.  D. Neumann and J. Wahl, \newblock  Universal abelian covers of quotient-cusps, {\em Math.  Ann.}  326, 75-93 (2003).
\bibitem{NW2}W.  D. Neumann and J. Wahl, \newblock The End Curve theorem for normal complex surface singularities,   {\em J. Eur. Math. Soc.} (JEMS) 12 (2010), no. 2, 471-503.
\bibitem{Nikulin} V. V. Nikulin, \newblock Integral symmetric bilinear forms and some of their applications, Engl. translations in {\em Math. USSR Izv.} 14, No.1 (1980), 103-167.
\bibitem{Pinkham}H. C. Pinkham, \newblock Deformation of algebraic varieties with $\Gm$-action, {\em Ast\'erisque} 20 (1974), 1-131. 
\bibitem{Pinkham2}H. C. Pinkham, \newblock  Automorphisms of cusps and Inoue-Hirzebruch surfaces, {\em Compositio Math.} 52, No.3 (1984), 299-313. 
\bibitem{PS}Y. Prokhorov and C. Shramov, \newblock Automorphism groups of Inoue and Kodaira surfaces, arXiv:1812.02400. 


\bibitem{Shepherd-Barron}  N. I. Shepherd-Barron, \newblock  Degenerations with numerically effective canonical divisor, {\em The birational geometry of degenerations}, Progress in Mathematics, Vol. 29, 33-84.
\bibitem{Shepherd-Barron2}  N. I. Shepherd-Barron, \newblock Extending polarizations on families of K3 surfaces, In: {\em The birational geometry of degenerations}, Progress in Mathematics, 1983. 
\bibitem{Simonetti} A. Simonetti, \newblock  $\zz_2$-equivariant smoothings of cusps singularities, arXiv:2201.02871.

\bibitem{Steenbrink} J. H. M. Steenbrink, \newblock Mixed Hodge structure associated with isolated singularities, {\em Proc. Symp. Pure Math.} 40, Part 2 (1983), 513-536.  
\bibitem{Wahl}J. Wahl,  \newblock Smoothing of normal surface singularities, {\em Topology}, Vol. 20, 219-246 (1981).
\bibitem{Wahl2}J. Wahl,  \newblock Derivations of negative weight and non-smoothability of certain singularities, {\em Math. Ann.}, 258, 383-398, (1982).
\bibitem{Wahl3}J. Wahl,  \newblock Simultaneous resolution and discriminantal loci, {\em Duke Math. J.}, Vol. 46, No.2, 341-375, (1979).
\bibitem{Wall} C. T. C. Wall,  \newblock Quadratic forms on finite groups, and related topics, {\em Topology} 2 (1963), 281-298.
\bibitem{Wall2} C. T. C. Wall,  \newblock  Quadratic forms on finite groups II,  {\em  Bull. London Math. Soc.} 4 (1972), 156-160.
\end{thebibliography}
\end{document}